\documentclass[11pt,reqno]{amsart}

\usepackage[margin=1.1in]{geometry}
\usepackage{amsmath,amssymb,amsthm}
\usepackage{array}
\usepackage{microtype}
\usepackage{mathtools}
\usepackage[dvipsnames]{xcolor}
\usepackage[colorlinks=true,linkcolor=NavyBlue,citecolor=NavyBlue,urlcolor=NavyBlue]{hyperref}

\theoremstyle{plain}
\newtheorem{theorem}{Theorem}[section]
\newtheorem{lemma}[theorem]{Lemma}

\newtheorem{conjecture}[theorem]{Conjecture}
\newtheorem{corollary}[theorem]{Corollary}

\theoremstyle{definition}

\newtheorem*{note*}{Note}

\theoremstyle{remark}

\newcommand{\F}{\mathbb{F}}
\newcommand{\R}{\mathbb{R}}

\newcommand{\norm}[1]{\left\lVert #1 \right\rVert}

\numberwithin{equation}{section}

\title{Kusner's conjecture is false for $p>4$}

\author{Nathan Xiong}
\address{Department of Mathematics, Massachusetts Institute of Technology, Cambridge, MA 02139, USA}
\email{nxiong@mit.edu}

\begin{document}

\begin{abstract}
An equilateral set is a set of points in a metric space whose pairwise distances are all equal. Kusner conjectured that the maximum cardinality of an equilateral set in $\R^n$ with the $\ell_p$ metric is $n+1$ for every $1<p<\infty$. We disprove the conjecture for all $p>4$. In particular, for every such $p$, we define $m=m(p)$ and construct an equilateral set of $8m$ points in $\R^{8m-2}$. Previously, Swanepoel disproved the conjecture for $1<p<2$, while Ge, Xu, and Zhou showed that it holds for $2\le p\le 4$. Combining these prior results with the paper's main theorem resolves Kusner's conjecture for all $1<p<\infty$. 
\end{abstract}

\maketitle

\section{Introduction}\label{sec:intro}

A subset $S$ of a metric space $X$ is called \emph{equilateral} if all pairwise distances between distinct points of $S$ are equal, and the \emph{equilateral dimension} $e(X)$ is the maximum cardinality of an equilateral subset of $X$. The study of equilateral sets is a classical problem in discrete geometry.

\subsection{Kusner's conjecture}
For $1\le p<\infty$, let $\ell_p^n$ denote $\R^n$ equipped with the norm
\[\norm{x}_p=\left(\sum_{i=1}^{n}|x_i|^p\right)^{1/p}.\]

Kusner \cite{Guy} posed the following conjecture on the equilateral dimension of $\ell_p^n$.
\begin{conjecture}[Kusner's conjecture]\label{conj:kusner}
    $e(\ell_1^n)=2n$ and $e(\ell_p^n)=n+1$ for all $1<p<\infty$. 
\end{conjecture}
For $\ell_1^n$, note that the set $\{\pm e_i\colon i=1,\ldots,n\}$, where $e_i$ is the $i$-th standard basis vector, forms an equilateral set, so $e(\ell_1^n)\ge 2n$. The strongest known upper bound is $e(\ell_1^n)\le cn\log n$ for some absolute constant $c$, proven by Alon and Pudl\'ak \cite{AP03}. For $\ell_p^n$ when $p>1$, the standard basis vectors $e_i$, together with $\alpha\sum_{i=1}^{n}e_i$ for a suitably chosen $\alpha$, form an equilateral set, so $e(\ell_p^n)\ge n+1$. Smyth \cite{smyth1} was the first to obtain a nontrivial upper bound on $e(\ell_p^n)$ for arbitrary $n$ and $p$, and it was later improved by Alon and Pudl\'ak \cite{AP03} to $e(\ell_p^n)\le c_p n^{(2p+2)/(2p-1)}$ for some constant $c_p$ depending on $p$.

Kusner's conjecture has already been resolved for a subset of $p$ values.
\begin{itemize}
    \item For $p=2$, it is easy to show that Conjecture~\ref{conj:kusner} holds.
    \item For $p=4$, Swanepoel \cite{Swanepoel2} proved that Conjecture~\ref{conj:kusner} holds.
    \item For $p\in(1,2)$, Swanepoel \cite{Swanepoel2} showed that there always exists an equilateral set with size greater than $n+1$ for a sufficiently large $n$, thus disproving Conjecture~\ref{conj:kusner}.
    \item For $p\in(2,4)$, Ge, Xu, and Zhou \cite{ge2026} proved that Conjecture~\ref{conj:kusner} holds.
\end{itemize}
Thus, Kusner's conjecture has been resolved for all $1<p\le 4$. Meanwhile, various improved upper bounds on $e(\ell_p^n)$ have been proven (see \cite{Konyagin2011}, \cite{CHEN2022103459}, and \cite{ge2026}). Furthermore, Chalmers \cite{chalmers2026} recently constructed a counterexample to Conjecture~\ref{conj:kusner} for $p=5$ (that extends to all $p$ in a small interval around 5), but this paper presents a different construction that works for the entire $p>4$ regime.

\subsection{Main result}
The paper's main result disproves Conjecture~\ref{conj:kusner} for all $p>4$. 
\begin{theorem}\label{thm:main}
For any $p>4$, there exists a positive integer $m$ and an equilateral set of $8m$ points in  $\ell_p^{8m-2}$.
\end{theorem}
\begin{corollary}\label{cor:kusner}
Let $1<p<\infty$. Then $e(\ell_p^n)=n+1$ holds for all $n\ge 1$ if and only if $2\le p\le 4$.
\end{corollary}

\section{Equilateral set construction}\label{sec:construction}
Fix $p>4$ and let $m=2^k$ for some positive integer $k$, to be defined later. In this section, we construct an equilateral set $S$ of $8m$ points in $\ell_p^{8m-2}$. To specify this set, we decompose every point in $S$ into its ``front'' $4m$ coordinates and its ``back'' $4m-2$ coordinates.

\subsection{Front coordinates}
Let $a>1$ be some parameter to be chosen later and define the following four 4-dimensional vectors, indexed by $\F_2^2$:
\begin{align*}
    q_{00}&=(a,1,1,0) \\
    q_{10}&=(-1,a,0,1) \\
    q_{01}&=(-1,0,a,-1) \\
    q_{11}&=(0,-1,1,a)
\end{align*}
First, note that all four $q_i$ vectors have the same $\ell_p$-norm. Hence, we can define $R\coloneqq \norm{q_i}_p^p=a^p+2$. Moreover, if we define
\begin{align*}
    A &\coloneqq (a+1)^p+(a-1)^p+2 \\
    B &\coloneqq 2a^p+2^p
\end{align*}
then it is not hard to check that for all distinct $i,i'\in\F_2^2$, we have
\begin{equation}\label{eq:sign-blind}
\norm{q_i-q_{i'}}_p^p=\norm{q_i+q_{i'}}_p^p=
\begin{cases}
A & \text{if } i+i'=10\text{ or }01,\\
B & \text{if } i+i'=11 .
\end{cases}
\end{equation}

\begin{lemma}\label{lem:AB}
For $p>2$ and $a>1$, we have $A>B$.
\end{lemma}

\begin{proof}
Write $A-B=(a+1)^p+(a-1)^p-2a^p-(2^p-2)$. This vanishes at $a=1$, and its derivative with respect to $a$ is $p\cdot\left((a+1)^{p-1}+(a-1)^{p-1}-2a^{p-1}\right)$, which is positive for $a>1$ and $p>2$ due to the strict convexity of $x\mapsto x^{p-1}$ on $[0,\infty)$.
\end{proof}

Now, for every $(i,s)\in\F_2^2\times\F_2^k$, define $U_{i,s}\in\R^{4m}$ as the concatenation of $m$ $4$-dimensional vectors, indexed by $r\in\F_2^k$, where the vector with index $r$ is $(-1)^{r\cdot s}q_i$. Equivalently, let $H_k$ be the Hadamard matrix with rows and columns indexed by $\F_2^k$, and let $(H_k)_{s,:}$ refer to the $s$-th row of $H_k$. Then, $U_{i,s}=(H_k)_{s,:}\otimes q_i$. By definition of $R$, we have $\norm{U_{i,s}}_p^p=mR$ for all $(i,s)$.

\subsection{Back coordinates}
Consider the Hadamard matrix $H_{k+2}$ with rows and columns indexed by $\F_2^2\times\F_2^k$: its entry in row $(i,s)$ and column $(j,t)$ is $(-1)^{i\cdot j+s\cdot t}$. Within row $(i,s)$, the four columns with $t=0$ depend only on $i$. Letting $i=(i_0,i_1)$, the four column entries are
\begin{align*}
    (j,t)=(00,0) &\quad\longrightarrow\quad 1\\
    (j,t)=(10,0) &\quad\longrightarrow\quad (-1)^{i_0}\\
    (j,t)=(01,0) &\quad\longrightarrow\quad (-1)^{i_1}\\
    (j,t)=(11,0) &\quad\longrightarrow\quad (-1)^{i_0+i_1}
\end{align*}
Let $\alpha,\beta>0$ be constants, to be determined later. For every $(i,s)\in\F_2^2\times \F_2^k$, let $V_{i,s}\in\R^{4m-2}$ be the row $(H_{k+2})_{(i,s),:}$ with the two entries corresponding to columns $(00,0)$ and $(11,0)$ removed, then the two entries corresponding to columns $(10,0)$ and $(01,0)$ multiplied by $\beta$, and the remaining $4m-4$ entries multiplied by $\alpha$. Explicitly, $V_{i,s}$ has the $4m-4$ coordinates
\[\alpha\,(-1)^{i\cdot j+s\cdot t}\qquad \forall (j,t)\in\F_2^2\times\left(\F_2^k\setminus\{0\}\right),\]
concatenated with the two extra coordinates $\beta(-1)^{i_0}$ and $\beta(-1)^{i_1}$.

\subsection{Full construction}
Let $S\subset \R^{8m-2}$ be the set of $8m$ points defined as
\[S=\left\{X_{\sigma,i,s}\coloneqq (\sigma U_{i,s},V_{i,s})\colon \sigma\in\{\pm 1\}, (i,s)\in\F_2^2\times\F_2^k\right\}\]
We choose $a>1$ and $m=2^k$ so that they satisfy the equation
\begin{equation}\label{eq:constraint}
2A-B = 2^{p-1}R\left(1+\frac{1}{m}\right)
\end{equation}
and prove that such a solution always exists in Section~\ref{sec:existence}. Finally, we define $\alpha$ and $\beta$ via
\begin{equation}\label{eq:alphabeta}
\alpha^{p}=\frac{R}{4}\qquad\text{ and }\qquad\beta^{p}=\frac{m(A-B)}{2^p}
\end{equation}
where Lemma~\ref{lem:AB} ensures that $\beta$ is well-defined.

\section{Pairwise distance computation}\label{sec:distances}
Throughout this section, consider two distinct points $X_{\sigma,i,s},X_{\sigma',i',s'}\in S$. We decompose the $\ell_p$-distance into the front and back contributions.
\[\norm{X_{\sigma,i,s}-X_{\sigma',i',s'}}_p^p=\norm{\sigma U_{i,s}-\sigma'U_{i',s'}}_p^p+\norm{V_{i,s}-V_{i',s'}}_p^p.\]
We split the possibilities into four cases:
\begin{enumerate}
\item $i=i'$ and $s=s'$
\item $i=i'$ and $s\neq s'$
\item $i+i'=10\text{ or }01$
\item $i+i'=11$
\end{enumerate}

\subsection{Front contribution}
The front contribution is $\norm{\sigma U_{i,s}-\sigma' U_{i',s'}}_p^p$.

\subsubsection*{Case 1.}
Recall that $\norm{U_{i,s}}_p^p=mR$ for all $(i,s)$. Since $(i,s)=(i',s')$, we must have $\sigma=-\sigma'$, and thus the front contribution is $\norm{\pm 2U_{i,s}}_p^p=2^pmR$.

\subsubsection*{Case 2.}
Since $i=i'$, the $r$-th vector of $\sigma U_{i,s}-\sigma'U_{i',s'}$ is $\varepsilon q_i$ where $\varepsilon=\sigma (-1)^{r\cdot s}-\sigma'(-1)^{r\cdot s'}$. It is well known that for any two distinct rows in a Hadamard matrix $H_k$, their column entries agree in exactly $2^{k-1}$ positions and differ in the other $2^{k-1}$ positions. Since $s\neq s'$, it follows that $(-1)^{r\cdot s}$ and $(-1)^{r\cdot s'}$ agree for exactly $2^{k-1}$ values of $r$ and differ for the other $2^{k-1}$ values. Hence, if $\sigma=\sigma'$, then $\varepsilon=\pm 2$ for the $2^{k-1}$ values of $r$ where they differ and $\varepsilon=0$ for the $2^{k-1}$ values of $r$ where they agree. On the other hand, if $\sigma\neq\sigma'$, then $\varepsilon=\pm 2$ for the $2^{k-1}$ values of $r$ where they agree and $\varepsilon=0$ for the $2^{k-1}$ values of $r$ where they differ. Either way, the front contribution is
\[\norm{\sigma U_{i,s}-\sigma' U_{i',s'}}_p^p= 2^{k-1}\cdot 2^pR=2^{p-1}mR.\]

\subsubsection*{Case 3.}
Now, we have $i\neq i'$, and we write the $r$-th vector of $\sigma U_{i,s}-\sigma' U_{i',s'}$ as $\pm (q_i\pm q_{i'})$ for some choice of signs. By \eqref{eq:sign-blind}, the front contribution is $mA$. 

\subsubsection*{Case 4.}
Again, we have $i\neq i'$, and by \eqref{eq:sign-blind}, the front contribution is $mB$.

\subsection{Back contribution}
The back contribution is $\norm{V_{i,s}-V_{i',s'}}_p^p$.

\subsubsection*{Case 1.}
Since $(i,s)=(i',s')$, the back contribution is just 0.

\subsubsection*{Case 2.}
Now, we have $(i,s)\neq (i',s')$, so the points $V_{i,s}$ and $V_{i',s'}$ correspond to two distinct rows in the Hadamard matrix $H_{k+2}$. The two rows agree in exactly $2^{k+1}$ positions and differ in the other $2^{k+1}$ positions. In particular, the entries in column $(j,t)$ differ if and only if
\begin{equation}\label{eq:hadamard-cond}
(i+i')\cdot j+(s+s')\cdot t=1
\end{equation}
Since we assume $i=i'$, this reduces to $(s+s')\cdot t=1$. Note that we actually take $t\in\F_2^k\setminus\{0\}$ (rather than $t\in\F_2^k)$ for our definition of $V_{i,s}$, but $t=0$ cannot satisfy this equation anyway. Hence, there are exactly $2^{k+1}=2m$ positions where the two rows differ, and the entries differ by $\pm 2\alpha$. Thus, the back contribution is
\[\norm{V_{i,s}-V_{i',s'}}_p^p=2m(2\alpha)^p.\]

\subsubsection*{Case 3.}
As in the previous case, $(i,s)$ and $(i',s')$ correspond to two distinct rows in the Hadamard matrix $H_{k+2}$. This time though, we need to consider the additional solutions to \eqref{eq:hadamard-cond} when $t=0$. In particular, when $t=0$ and $i+i'=10\text{ or }01$, there are exactly two valid $j$ values. Hence, there are $2^{k+1}-2=2m-2$ positions among the $(j,t)\in\F_2^2\times \left(\F_2^k\setminus\{0\}\right)$ columns where the entries differ, and they differ by $\pm 2\alpha$. Moreover, since $i+i'=10\text{ or }01$, exactly one of the $\beta(-1)^{i_0}$ and $\beta(-1)^{i_1}$ entries differs (and the other agrees), and the difference is $\pm 2\beta$. Hence, the back contribution is
\[\norm{V_{i,s}-V_{i',s'}}_p^p=(2m-2)(2\alpha)^p+(2\beta)^p.\]

\subsubsection*{Case 4.}
Using the same argument as in the previous case, there are $2m-2$ positions among the $(j,t)\in\F_2^2\times\left(\F_2^k\setminus\{0\}\right)$ columns where the entries differ, and they differ by $\pm 2\alpha$. Then, since $i+i'=11$, both the $\beta(-1)^{i_0}$ and $\beta(-1)^{i_1}$ entries differ, and the difference is $\pm 2\beta$. Hence, the back contribution is
\[\norm{V_{i,s}-V_{i',s'}}_p^p=(2m-2)(2\alpha)^p+2(2\beta)^p.\]

\subsection{Total contribution}
To finish the computation, it suffices to prove that the total contribution (sum of the front and back contributions) is the same for all four cases.

\subsubsection*{Case 1.}
The total contribution is just $2^pmR$. 

\subsubsection*{Case 2.}
Plugging in the definition of $\alpha$ from \eqref{eq:alphabeta}, the total contribution is
\[2^{p-1}mR+2m(2\alpha)^p=2^{p-1}mR+2m\cdot 2^{p-2}R=2^pmR.\]

\subsubsection*{Case 3.}
Plugging in the definitions of $\alpha$ and $\beta$ from \eqref{eq:alphabeta}, as well as the constraint in \eqref{eq:constraint}, the total contribution is
\begin{align*}
    mA+(2m-2)(2\alpha)^p+(2\beta)^p &= mA+(2m-2)2^{p-2}R+m(A-B)\\
    &= m(2A-B)+2^{p-1}mR-2^{p-1}R\\
    &= 2^{p-1}mR+2^{p-1}R+2^{p-1}mR-2^{p-1}R\\
    &= 2^pmR
\end{align*}

\subsubsection*{Case 4.}
Plugging in the definitions of $\alpha$ and $\beta$ from \eqref{eq:alphabeta}, as well as the constraint in \eqref{eq:constraint}, the total contribution is
\[mB+(2m-2)(2\alpha)^p+2(2\beta)^p=mB+(2m-2)2^{p-2}R+2m(A-B)=2^pmR\]

\section{Proof of existence}\label{sec:existence}
To complete the proof of Theorem~\ref{thm:main}, it suffices to show that for every $p>4$, there exists an $a>1$ and $m=2^k$ satisfying \eqref{eq:constraint}.
\begin{lemma}\label{lem:power-ineq}
For $p>4$ and distinct $x,y>0$,
\[(x+y)^p+|x-y|^p>2x^p+2y^p+(2^p-4)(xy)^{p/2}.\]
\end{lemma}

\begin{proof}
By homogeneity and symmetry, it suffices to show that, for all $0<t<1$,
\[Q(t)\coloneqq\frac{(1+t)^p+(1-t)^p-2-2t^p}{t^{p/2}}>Q(1).\]
The derivative of $Q$ is
\[Q'(t)=\frac{p}{t^{p/2+1}}\left[1-t^p-\frac{1-t^2}{2}\left((1+t)^{p-2}+(1-t)^{p-2}\right)\right].\]
Let $q=\frac{p-2}{2}>1$. Apply strict Jensen's inequality to $u\mapsto u^q$
with weights $\frac{1-t^2}{2}$, $\frac{1-t^2}{2}$, and $t^2$
and corresponding arguments $(1+t)^2$, $(1-t)^2$, and $t^2$. The weighted average of the arguments is
\[\frac{1-t^2}{2}\left((1+t)^2+(1-t)^2\right)+t^4=1.\]
Thus,
\[\frac{1-t^2}{2}\left((1+t)^{p-2}+(1-t)^{p-2}\right)+t^p>1.\]
It follows that $Q'(t)<0$ and $Q(t)>Q(1)$, as desired.
\end{proof}

\begin{lemma}
Let $p>4$. For any sufficiently large $m$, there exists $a>1$ such that, with $R\coloneqq a^p+2$, $A\coloneq (a+1)^p+(a-1)^p+2$, and $B\coloneq 2a^p+2^p$, we have
\[2A-B=2^{p-1}R\left(1+\frac{1}{m}\right).\]
\end{lemma}

\begin{proof}
Set $a_0=4^{1/p}>1$. Applying Lemma~\ref{lem:power-ineq} with
$x=a_0$ and $y=1$ gives
\[(a_0+1)^p+(a_0-1)^p>2^{p+1}+2.\]
Thus, at $a=a_0$, we have $R=6$ and
\begin{align*}
    2A-B &= 2(a_0+1)^p+2(a_0-1)^p+4-2a_0^p-2^p\\
    &> 3\cdot 2^p\\
    &= 2^{p-1}R
\end{align*}
Define the continuous function $\Phi$ over $a\in[1,a_0]$ by
\[\Phi(a)=\frac{2A-B}{2^{p-1}R}.\]
We already showed that $\Phi(a_0)>1$. Then, we can plug in $a=1$ to get
\[\Phi(1)=\frac{2^p+2}{3\cdot2^{p-1}}<1.\]
Therefore, as long as $m>\frac{1}{\Phi(a_0)-1}$, the intermediate value theorem supplies an $a\in(1,a_0)$ such that
$\Phi(a)=1+\frac{1}{m}$.
\end{proof}

\section{Acknowledgements}
The construction was found by OpenAI's GPT-6 Astra model. The author is responsible for any correctness issues. 

\bibliographystyle{amsplain}
\bibliography{bibfile}

\end{document}